\documentclass[12pt,a4paper]{amsart}
\usepackage[top=35mm, bottom=30mm, left=30mm, right=30mm]{geometry}
\usepackage{mathptmx}
\usepackage{mathrsfs}

\usepackage{verbatim}
\usepackage{url}
\usepackage[all]{xy}
\usepackage{xcolor}
\usepackage[colorlinks=true,citecolor=blue]{hyperref}

\usepackage{amsmath,amsthm}
\usepackage{amssymb}

\usepackage{enumerate}

\usepackage{graphicx}

\newtheorem{thm}{Theorem}[section]
\newtheorem{cor}[thm]{Corollary}
\newtheorem{lem}[thm]{Lemma}
\newtheorem{prop}[thm]{Proposition}

\newtheorem{ques}[thm]{Question}

\theoremstyle{definition}

\numberwithin{equation}{section}

\begin{document}
\title[The nonexistence of expansive actions  ]{The nonexistence of expansive actions on Suslinian continua by groups of subexponential growth}

\author[B.B.~Liang]{Bingbing Liang}
\address{Soochow University, Suzhou, Jiangsu 215006, China}
\email{bbliang@suda.edu.cn}

\author[E.H.~Shi]{Enhui Shi}
\address{Soochow University, Suzhou, Jiangsu 215006, China}
\email{ehshi@suda.edu.cn}

\author[Z.W.~Xie]{Zhiwen Xie}
\address{Soochow University, Suzhou, Jiangsu 215006, China}
\email{20204007002@stu.suda.edu.cn}

\author[H.~Xu]{Hui Xu}
\address{Soochow University, Suzhou, Jiangsu 215006, China}
\email{20184007001@stu.suda.edu.cn}

\begin{abstract}
We show that if $G$ is a finitely generated group of subexponential growth and $X$ is a Suslinian continuum, then any action of $G$ on $X$ cannot be expansive. 
\end{abstract}

\maketitle

\date{\today}

\section{Introduction}

  \medskip
 By a \textit{continuum} we mean a nonempty compact connected metric space. A continuum $X$ is said to be \textit{nondegenerate} if it is not a single point.
If a continuum $X$ does not contain uncountably many mutually disjoint nondegenerate subcontinua, then $X$ is called \textit{Suslinian}. It is known that all
regular curves are Suslinian.

 \medskip
 Let $X$ be a compact metric space and  $\text{Homeo}(X)$ be the homeomorphism group of $X$. Let $G$ be a discrete group. A group homomorphism $\phi: G\rightarrow \text{Homeo}(X)$ is called a \textit {continuous action} of $G$ on $X$, which is denoted by $(X,G,\phi)$. For brevity, we usually use $gx$ or $g(x)$ in place of $\phi(g)(x)$.

\medskip
Let $(X,G,\phi)$ be a continuous action of $G$ on a compact metric space $(X,d)$. The action is said to be \textit{expansive} if there is a $c>0$ such that for any distinct points $x,y\in X$, $\sup_{g\in G}d(gx,gy)>c$. Such $c$ is called an {\textit {expansive constant}}.

 \medskip
 We are interested in the following question.

 \medskip
\begin{ques}
Given a compact metric space $X$ and a discrete group $G$, can $G$ act on $X$ expansively?
\end{ques}

 There has been intensively studied around this question. One may refer to the introduction part of \cite{WSXX} for a detailed description of the progress.

 \medskip
 The following theorem is due to H. Kato.

  \medskip
\begin{thm}[\cite{Kato}]
There are no expansive $\mathbb Z$-actions on Suslinian continua.
\end{thm}

 \medskip
 The purpose of the paper is to extend Kato's theorem to actions by groups of subexponential growth.
Let $H$ be a finitely generated group and $S$ be a finite set of generators. For each $h\in H$,  denote by $|h|$ the word length with respect to $S$. 
For $k\in \mathbb{N}$, define
\[\beta(H,S; k)= \#\{h\in H:~|h|\leq k\},\]
 which is called the \textit{growth function} of $H$ with respect to $S$. If
\[\lim_{k\rightarrow \infty}\sqrt[k]{\beta(H,S;k)}\leq1,\] then $H$ is said to be of \textit{subexponential growth}.
One may consult \cite[Chapter 6]{CC} for the  example of a group which is of subexponential growth but not of polynomial growth.

\medskip
The following is the main theorem of the paper.

\medskip
\begin{thm}\label{main thm}
Let $G$ be a finitely generated group of subexponential growth and $X$ be a Suslinian continuum. Then $G$ cann't act on
$X$ expansively.
\end{thm}

\medskip
Since every finitely generated nilpotent group is of subexponential growth, the following corollary is immediate.

\medskip
\begin{cor}
Let $G$ be a finitely generated nilpotent group and $X$ be a Suslinian continuum. Then $G$ cann't act on
$X$ expansively.
\end{cor}

\medskip
The proof of the main theorem relies on a comparison between the growth rate of pairwise disjoint subcontinua of unform scales and 
the growth rate of the acting group. Meanwhile, we use a key lemma established by T. Meyerovitch and M. Tsukamoto, and follows the same line as in H. Kato's proof
for $\mathbb Z$-actions. 

\section{Preliminaries}

\subsection{Meyerovitch-Tsukamoto's Lemma}\

\medskip
The following Proposition is due to T. Meyerovitch and M. Tsukamoto.
\medskip
\begin{prop}~\cite[Lemma 4.4]{MT}\label{compatible metric}
Let $G$ be a finitely generated discrete group and $\phi:G\rightarrow \text{Homeo}(X)$ be a continuous action of $G$ on a compact metric space $(X,d)$. If the action is expansive, then there exist $\alpha>1$ and a compatible metric $D$ on $X$ such that for any positive integer $n$ and any two distinct points $x,y\in X$ satisfying $D(x,y)\geq \alpha^{-n}$, we have
\[\max_{g\in G, |g|\leq n} D(gx,gy)\geq \frac{1}{4\alpha}.\]
\end{prop}
Here we remark that although the above proposition is established only for $\mathbb{Z}^k$-actions in \cite[Lemma 4.4]{MT}, a careful checking will
 show that it holds for any finitely generated group action and the proof is completely the same. For the convenience  of the readers, we add
 the proof in the appendix.

\subsection{Kato's characterization of Suslinian continuum}\

\medskip
Let $X$ be a continuum.  The {\it hyperspace} $C(X)$  is the set of all subcontinua of $X$.  For  $A,B\in C(X)$, define
\[d_H(A,B)=\inf\{\delta>0:~A\subset N_{\delta}(B)~~\text{and}~~B\subset N_{\delta}(A)\},\]
where $N_{\delta}(A)$ denotes the $\delta$-neighborhood of $A$ in $X$. Then $d_H$ is a metric on $C(X)$ and is called the {\it Hausdorff metric}.
It is known that  $(C(X), d_H)$ is a continuum \cite[Chapter IV]{Nadler}.

\medskip
For any subset $M$ of $C(X)$, define
\begin{eqnarray*}
\widetilde{M}=\{ A\in C(X)&:& \text{for any }~~ \varepsilon>0~~\text{and}~~ k\in\mathbb{N},~~\text{there exist pairwise disjoint nondegenerate }\\ &&
\text{ subcontinua}~~A_1,A_2,\cdots,A_k\in M~~
\text{such that}~~ d_{H}(A,A_i)<\varepsilon\}.
\end{eqnarray*}
Set $M_1= C(X)$. For every ordinal $\alpha$, let $M_{\alpha+1}=\widetilde{M_{\alpha}}$; for every limit ordinal $\lambda$, let $M_{\lambda}=\bigcap_{\alpha<\lambda}M_{\alpha}$.
By the definition, we see that the family $\{M_\alpha\}$ is decreasing with respect to $\alpha$.
\medskip
In \cite{Kato}, Kato proved the following theorem as a characterization of Suslinian continua.

\begin{thm}\cite[Theorem 3.4]{Kato}\label{suslin}
A continuum $X$ is Suslinian if and only if $M_{\alpha}=\emptyset$ for some countable ordinal $\alpha$.
\end{thm}

\medskip
If $(X,G,\phi)$ is a continuous action on continuum $X$ by a group $G$ , then this action naturally induces a continuous action
$(C(X),G,\tilde{\phi})$ by defining $\tilde{\phi}(g)(A)=\phi(g)(A)$ for $g\in G$ and $A\in C(X)$.

\medskip
By the definition of $M_\alpha$, we immediately have

\begin{prop}\label{inv}
Every $M_\alpha$ is a $G$-invariant compact subset of $C(X)$ under the action $(C(X),G,\tilde{\phi})$.
\end{prop}

\subsection{Pairwise disjoint subcontinua of uniform scales}\

\medskip
The following lemma is known as the Boundary Bumping Theorem.

\medskip
\begin{lem}\cite[Theorem 5.4]{Nadler}\label{bumping thm}
Let $X$ be a continuum, and let $U$ be a nonempty, proper, open subset of $X$. If $K$ is a connected component of $\overline{U}$, then $K\cap Bd(U)\neq\emptyset$.
\end{lem}

\medskip
Let $(X,d)$ be a compact metric space.
For a subset $E$ of $X$ and $\varepsilon>0$, we say $E$ is $\varepsilon$-\textit{separated} if for any distinct $x,y\in X$, $d(x,y)\geq \varepsilon$. Let $S(\varepsilon)$ denote the cardinality of a maximal $\varepsilon$-separated subset of $X$. The \textit{lower box dimension} of $X$ is defined to be

\begin{equation*}
\underline{\dim}_B(X,d)=\liminf_{\varepsilon\rightarrow 0} \frac{\log S(\varepsilon)}{\log(1/\varepsilon)}.
\end{equation*}

\begin{lem}\label{disjoint subcontinua}
Let $X$ be a nondegenerate continuum. Then there exists $\varepsilon_0>0$ such that for any $\varepsilon \in (0,\varepsilon_0)$, there are more than $1/\sqrt{\varepsilon}$ pairwise disjoint subcontinua whose diameters $\geq \varepsilon/3$.
\end{lem}
\begin{proof}
It is well known that the topological dimension of a nondegenerate continuum is no less than one (\cite[Proposition 1.3.3]{Coornaert}) and does not exceed its lower box dimension (\cite[Chapter VII]{HW}). Thus
\[ \underline{\dim}_{B}(X)=\liminf_{\varepsilon\rightarrow 0}\frac{\log S(\varepsilon)}{\log (1/\varepsilon)}\geq 1.\]
Then there exists $\varepsilon_0>0$ such that for any $\varepsilon \in (0,\varepsilon_0)$, $S(\varepsilon)> \frac{1}{\sqrt{\varepsilon}}$. Fix an $\varepsilon$-separated subset $\{x_1,\cdots, x_{S(\varepsilon)}\}$ of $X$. Thus the closed balls $\overline{B(x_1,\varepsilon/3)}, \cdots, \overline{B(x_{S(\varepsilon)},\varepsilon/3)}$ are disjoint. For each $i\in\{1,2,\cdots,S(\varepsilon)\}$, let $A_i$ be the component of $\overline{B(x_i,\varepsilon/3)}$ containing $x_i$. By Lemma \ref{bumping thm}, we have $A_i\cap Bd (B(x_i,\varepsilon/3))\neq\emptyset$. Hence $\text{diam}(A_i)\geq \frac{\varepsilon}{3}$. Thus these $A_i$'s satisfy the requirements.
\end{proof}

\section{Proof of the main theorem}
\medskip
The following lemmas will  be used in the subsequent proof.
\begin{lem}\cite[Lemma 2.1]{Kato}\label{clustering}
Let $(Y,d)$ be a compact metric pace. For any $\varepsilon>0$ and positive integer $k$, there is a positive integer $n=n(\varepsilon, k)\geq k$ such that if $y_1,y_2,\cdots,y_n$ are points of $X$, then there exists a point $y\in Y$ and $1\leq i(1)<i(2)<\cdots<i(k)\leq n$ such that $d(y, y_{i(j)})<\varepsilon$ for $j\in\{1,2,\cdots,k\}$.
\end{lem}

\begin{lem}\cite[Lemma 2.2]{Kato}\label{boundary lemma}
Let $X$ be a compact metric space and let $U,V$ be open subsets of $X$ such that $\overline{V}\subseteq U$. If $A$ is a subcontinuum of $X$ such that $A\cap V\neq \emptyset$ and $A\setminus\overline{U}\neq\emptyset$, then there is a subcontinuum $B$ of $A\cap\overline{U}$ such that $B\cap V\neq \emptyset$ and $B\cap Bd(U)\neq\emptyset$.
\end{lem}

\begin{proof}[Proof of Theorem \ref{main thm}]
To the contrary, assume there is an expansive action of $G$ on $X$. Then, by Proposition \ref{compatible metric}, there exists $\alpha>1$ and a compatible metric $D$ on $X$ such that for any positive integer $n$ and any two distinct points $x,y\in X$ satisfying $D(x,y)\geq \alpha^{-n}$, we have
\[\max_{g\in G, |g|\leq n} D(gx,gy)\geq \frac{1}{4\alpha}.\]

Let $A$ be any nondegenerate subcontinuum of $X$. By Lemma \ref{disjoint subcontinua}, for sufficiently large $n$, there is a set $\mathcal{K}_n$ of disjoint subcontinua of $A$ such that $|\mathcal{K}_n|> \sqrt{\alpha^{n}/3}$ and for every $K\in\mathcal{K}_n$, $\text{diam}_D(K)\geq\alpha^{-n}$. By Lemma \ref{compatible metric}, for every $K\in\mathcal{K}_n$, there exists $g\in G$ with $|g|\leq n$ such that $\text{diam}_D(g(K))\geq\frac{1}{4\alpha}$.

Let $S$ be a finite set of generators of $G$. Since the growth function $\beta(G, S; n)$ is subexponential (i.e. $\lim_{n\rightarrow\infty}\sqrt[n]{\beta(G, S; n)}\leq1$),  for sufficiently large $n$, we have
\[\beta(G, S; n)\leq (\sqrt[3]{\alpha})^n.\]

By pigeonhole principle, we have the following claim.

\newtheorem*{cl}{Claim A}
\begin{cl}
For every sufficiently large $n$, there exists $g_n\in G$ with $|g_n|\leq n$ such that there are at least $a_n\triangleq\frac{\sqrt{\alpha^n/3}}{(\sqrt[3]{\alpha})^n}=\frac{\alpha^{n/6}}{\sqrt{3}}$ subcontinua of $\mathcal{K}_n$, the diameters of whose image under $g_n$ are no less than $\frac{1}{4\alpha}$.
\end{cl}

{\color{red}(The following argument is modelled on Kato's proof of Theorem 3.1 in \cite{Kato}.)}
\medskip

In the following, the underlying metric on $X$ is always $D$ .

Set $M=C(X)$ and choose a sequence $\varepsilon_1>\varepsilon_2>\cdots$ of positive numbers such that $\lim_{i\rightarrow\infty}\varepsilon_i=0$. For each $k$ and $\varepsilon_k$, take a positive number $n_k=n(\varepsilon_k,k)$ as Lemma \ref{clustering}, where we assume $Y=C(X)$.

Let $A$ be any nondegenerate subcontinuum of X. Note that $a_n\rightarrow\infty$ as $n\rightarrow\infty$.  Take an increasing sequence $m_1<m_2<m_3<\cdots$ of positive integers such that both $a_{m_k}\geq n_k$ and $m_k\geq n_k$ for each $k\in\mathbb{N}$. By  Claim A, for each $k\in\mathbb{N}$, we can choose pairwise disjoint nondegenerate subcontinua $B_1,B_2,\cdots,B_{m_k}$ of $A$ such that for each $i=1,2,\cdots,m_k$,
\begin{equation}\label{eq1}
\text{diam}(h_k(B_i))\geq \frac{1}{4\alpha},
\end{equation}
for some $h_k\in G$ with $|h_k|\leq m_k$. By the choice of $m_k$, there is a point $B(k)$ of $C(X)$ and $1\leq i_1<i_2<\cdots<i_k\leq m_k$ such that
\begin{equation}\label{eq2}
d_{H}(B(k), h_k(B_{i_j}))<\varepsilon_k,~~~\text{for }~~j=1,2,\cdots,k.
\end{equation}

By the compactness of $C(X)$, we may assume that $\{B(k)\}$ converges to a point $A_1$ of $C(X)$. Then $\text{diam}(A_1)\geq \frac{1}{4\alpha}$, by (\ref{eq1}) and (\ref{eq2}). Moreover, it is clear that $A_1\in M_1=\widetilde{M}$. Hence $M_1$ contains a nondegenerate subcontinuum. Now we shall show that $M_1$ satisfies the following condition $(\Xi_1)$:
\newtheorem*{cond1}{Condition $(\Xi_1)$}
\begin{cond1}
If $C\in M_1$ is nondegenerate, then for any open sets $U,V$ of $X$ satisfying $\overline{V}\subset U,~ C\cap V\neq \emptyset$ and $C\setminus \overline{U}\neq\emptyset$, there exists $D\in M_1$ such that $D\cap\overline{V}\neq\emptyset,~D\subset C\cap\overline{U}$ and $D\cap Bd(U)\neq \emptyset$.
\end{cond1}
Since $C\in M_1$, for each $k$ we can choose disjoint nondegenerate $D_1,D_2,\cdots, D_{n_k}\in C(X)$ such that $d_{H}(C,D_i)<\varepsilon_k$ for each $i$. We may assume
that $D_i\cap\overline{V}\neq\emptyset$ and $D_i\setminus \overline{U}\neq\emptyset$ for each $i$. By Lemma \ref{boundary lemma}, for each $i=1,2,\cdots, n_k$, there exists a subcontinuum $E_i$ of $D_i$ such that $E_i\subset \overline{U},~E_i\cap V\neq \emptyset$ and $E_i\cap Bd(U)\neq\emptyset$. By Lemma \ref{clustering}, there is a point $E(k)$ of $C(X)$ and $1\leq i_1<i_2<\cdots<i_k\leq n_k$ such that $d_H(E(k), D_{i_j})<\varepsilon_k$ for each $j=1,2,\cdots,k$. Furthermore, we may assume that $\{E(k)\}$ converges to a point $D$ of $C(X)$. Then it is easy to see that $D\subset C\cap \overline{U},~D\cap\overline{V}\neq\emptyset$ and $E\cap Bd(U)\neq \emptyset$. Clearly, $D\in M_1$. Thus $M_1$ satisfies Condition $(\Xi_1)$.\\

For a countable ordinal $\lambda$, we  assume that for every $\alpha<\lambda$, $M_{\alpha}$ contains a nondegenerate subcontinuum and satisfies Condition $(\Xi_\alpha)$. We shall show that $M_{\lambda}$ has the same properties. We divide it into the following two cases.
\begin{itemize}
  \item [(I)] $\lambda=\alpha+1$. Note that $M_{\alpha}$ satisfies Condition $(\Xi_\alpha)$. By the argument similar to the above, we can show that $M_{\lambda}$ contains a nondegenerate subcontinuum and satisfies Condition $(\Xi_\lambda)$.
  \item [(II)] $\lambda$ is a limit ordinal. In this case, take a sequence $\alpha_1<\alpha_2<\cdots$ of countable ordinals such that $\lim \alpha_i=\lambda$. Note that for each $\alpha$, $M_{\alpha}$ is $G$-invariant by Proposition \ref{inv}. Then by Claim A, for each $i$ there is $A_i\in M_{\alpha_i}$ such that $\text{diam}(A_i)\geq\frac{1}{4\alpha}$. Furthermore, we may assume that $\{A_i\}$ converges to a point $A_\lambda$ of $C(X)$. This implies
      \[ A_\lambda\in\bigcap_{\alpha<\lambda}M_{\alpha}
      =M_{\lambda}.\]
      Note that $\text{diam}(A_\lambda)\geq \frac{1}{4\alpha}$. By Lemma \ref{clustering}, we can prove that $M_{\lambda}$ satisfies Condition $(\Xi_\lambda)$.
\end{itemize}

Consequently, $M_{\alpha}\neq \emptyset$ for any contable ordinal $\alpha$. By Theorem \ref{suslin}, $X$ cannot be Suslinian. This is a contradiction. Hence we complete the proof.
\end{proof}

\section{Appendix}

\medskip
Now we write the proof of Proposition \ref{compatible metric} in the framework of general group actions, which is only a repeat of the proof of \cite[Lemma 4.4]{MT}.

\medskip
Fix an expansive constant $c>0$.

\medskip
\begin{lem}\label{uniform bound}
For any $\varepsilon>0$, there is an integer $n=n(\varepsilon)>0$ such that for any $x,y\in X$ satisfying $d(x,y)\geq \varepsilon$, we have
\[\max_{g\in G, |g|\leq n}d(gx,gy)\geq c.\]
\end{lem}
\begin{proof}
Assume the Lemma is not true. There exist $\varepsilon_0>0$ and $x_{k}, y_k\in X~(k\geq 1)$ satisfying
\[ d(x_k,y_k)\geq \varepsilon_0,~~ \max_{g\in G, |g|\leq k}d(gx_k,gy_k)< c.\]
By the compactness of $X$, we may assume that $x_k\rightarrow x$ and $y_k\rightarrow y$ as $k\rightarrow\infty$.
Then $d(x,y)\geq\varepsilon_0$ and $\sup_{g\in G}d(gx,gy)\leq c$. This contradicts the expansiveness. Thus we complete the proof.
\end{proof}
By Lemma \ref{uniform bound}, there is an integer $l>0$ such that for any $x,y\in X$ satisfying $d(x,y)\geq \frac{c}{2}$, we have $\max_{g\in G, |g|\leq l}d(gx,gy)\geq c$. Fix a real number $\alpha>1$ such that $\alpha^l<2$.

Let $x,y\in X$. If $x\neq y$, define
\[ n(x,y)=\min\{ n\in\mathbb{N}:~\exists g\in G, |g|\leq n ~~\text{such that}~~d(gx,gy)\geq c\}.\]
Otherwise, set $n(x,y)=\infty$. Set $\rho(x,y)=\alpha^{-n(x,y)}$.
\begin{lem}\label{rho}
The function $\rho$ satisfies:
\begin{itemize}
  \item [(1)] $\rho(x,y)=\rho(y,x)$;
  \item [(2)] $\rho(x,y)\geq 0$ and $\rho(x,y)= 0$ if and only if $x=y$;
  \item [(3)] $\rho(x,z)\leq 2\max\{\rho(x,y),~\rho(y,z)\}$;
  \item [(4)] if $d(x_k,x)\rightarrow 0$ and $d(y_k,y)\rightarrow 0$ as $k\rightarrow\infty$, then
      \[\limsup_{n\rightarrow\infty} \rho(x_k,y_k)\leq \rho(x,y);\]
  \item [(5)] $\rho$ is compatible with the topology of $X$, which means that the balls (w.r.t $\rho$)
  \[ B_{\rho}(x,r)=\{y\in X:~\rho(x,y)<r\}~~(x\in X,r>0)\] form an open base of the topology of $X$.
\end{itemize}
\end{lem}
\begin{proof}
It is clear for $(1)$ and $(2)$. For (3), we may assume $x\neq z$ and set $m=n(x,z)$. Then there exists $g\in G$ with $|g|\leq m$ such that $d(gx,gz)\geq c$. By the triangle inequality, either $d(gx,gy)\geq c/2$ or $d(gy,gz)\geq c/2$. We may assume the former. By the choice of $l$, there exists $h\in G$ with $|h|\leq l$ such that $d(hg(x),hg(y))\geq c$. Thus $n(x,y)\leq |hg|\leq m+l$. Then
\[\rho(x,y)=\alpha^{-n(x,y)}\geq \alpha^{-m}\alpha^{-l}>\frac{\rho(x,z)}{2},\]
by noting that $\alpha^l<2$.

(4) is equivalent to show $\liminf_{k\rightarrow \infty}n(x_k,y_k)\geq n(x,y)$. We may assume $\liminf_{k\rightarrow \infty}n(x_k,y_k)=m<\infty$, otherwise it is trivial.  By passing to some subsequence, we may assume that $\lim_{k\rightarrow \infty}n(x_k,y_k)=m$. Furthermore, we may assume that $n(x_k,y_k)=m$ for each $k$. Then for each $k$, there exists $g_{k}\in G, |g_k|\leq m$ such that $d(g_kx_k,g_ky_k)\geq c$. By uniform continuity, for any $\varepsilon>0$, there exists $\delta>0$ such for any $x,y\in X$ satisfying $d(x,y)<\delta$, we have for any $g\in G, |g|\leq m$, $d(gx,gy)<\varepsilon$. Noting that $\{x_k\}$ and $\{y_k\}$ are Cauchy sequences, there is $N>0$ such that for any $i,j\geq N$, $d(x_i,x_{j})<\delta$ and $d(y_i,y_{j})<\delta$; hence $d(gx_i,gx_j)<\varepsilon$ and $d(gy_i,gy_j)<\varepsilon$, for any $g\in G$ with $|g|\leq m$. Then for any $k\geq N$,
\begin{equation*}
d(g_{N}x_k, g_{N}y_k)\geq d(g_{N}x_N, g_{N}y_N)-d(g_{N}x_N, g_{N}x_k)-d(g_{N}y_N, g_{N}y_k)> c-2\varepsilon.
\end{equation*}
Letting $k$ tend to infinity we have $d(g_Nx,g_Ny)\geq c-2\varepsilon$. Thus $n(x,y)\leq m$, since $\varepsilon$ is arbitrary. So (4) holds.

On the one hand, the property (4) implies that all balls $B_{\rho}(x,r)$ are open with respect to the metric $d$. On the other hand,  by Lemma \ref{uniform bound}, for any $x\in X$ and $R>0$ there exists $r>0$ satisfying $B_{\rho}(x,r)\subset B_{d}(x, R)$. Thus (5) holds.
\end{proof}

In order to show the existence of the desired metric, we need the following Frink's metrization lemma.
\begin{lem}\cite[pp.134-135]{Frink}\cite[Theorem 4.1]{MT}\label{Frink}
Let $X$ be a set and $\rho$ be a nonnegative function on $X\times X$ satisfying
\begin{itemize}
  \item [(1)] $\rho(x,y)=\rho(y,x)$ for all $x,y\in X$;
  \item [(2)] $\rho(x,y)= 0$ if and only if $x=y$;
  \item [(3)] $\rho(x,z)\leq 2\max\{\rho(x,y),~\rho(y,x)\}$ for all $x,y,z\in X$.
\end{itemize}
Then there exists a distance function $D$ on $X$ satisfying for all $x,y\in X$,
\[\frac{1}{4}\rho(x,y)\leq D(x,y)\leq \rho(x,y).\]
\end{lem}
By the properties (1,2,3) of Lemma \ref{rho} and Lemma \ref{Frink}, there is a distance function $D$ on $X$ satisfying
$\frac{1}{4}\rho(x,y)\leq D(x,y)\leq \rho(x,y)$.
The property $(5)$ implies $D$ is a compatible metric on $X$.

\begin{proof}[Proof of Proposition \ref{compatible metric}]
It remains to show that the above metric $D$ satisfies that for any positive integer $n$ and any two distinct points $x,y\in X$ satisfying $D(x,y)\geq \alpha^{-n}$, we have
\[\max_{g\in G, |g|\leq n} D(gx,gy)\geq \frac{1}{4\alpha}.\]
Since $D(x,y)\leq \rho(x,y)$, we have $\rho(x,y)\geq \alpha^{-n}$ which implies that $n(x,y)\leq n$. Thus there exists $h\in G$ with $|h|\leq n$ such that $d(hx,hy)\geq c$. Hence $n(hx,hy)=1$. By the definition of $\rho$, $\rho(hx,hy)=\frac{1}{\alpha}$. Therefore,
\[\max_{g\in G, |g|\leq n} D(gx,gy)\geq D(hx,hy)\geq \frac{\rho(hx,hy)}{4}=\frac{1}{4\alpha}.\]
This completes the proof.
\end{proof}

\end{document}